\documentclass[12pt]{amsart}

\usepackage{amsmath,amssymb,amscd,amsthm}
\usepackage{hyperref}
\usepackage[left=2.5cm,right=2.5cm,top=2.5cm,bottom=2.5cm,includeheadfoot]{geometry}

\makeatletter
\@namedef{subjclassname@2010}{%
  \textup{2010} Mathematics Subject Classification}
\makeatother

\newtheorem{theorem}{Theorem}
\newtheorem{proposition}{Proposition}
\newtheorem {corollary}{Corollary}
\newtheorem {lemma}{Lemma}

\theoremstyle{definition}
\newtheorem{definition}{Definition}
\newtheorem{example}{Example}
\newtheorem{remark}{Remark}

\DeclareMathOperator\Sym{Sym}
\DeclareMathOperator\pt{point}
\DeclareMathOperator\Pfaffian{Pfaffian}

\begin{document}

\title[]{An identity involving symmetric polynomials and the geometry of Lagrangian Grassmannians}

\author[]{Dang Tuan Hiep}


\address{Faculty of Mathematics and Computer Science, Da Lat University, No. 1, Phu Dong Thien Vuong Rd., Ward 8, Da Lat, Vietnam}
\email{hiepdt@dlu.edu.vn}

\author[]{Nguyen Chanh Tu}
\address{Faculty of Advanced Science and Technology (FAST), Da Nang University of Science and Technology, 54 Nguyen Luong Bang, Da Nang, Vietnam}
\email{nctu@dut.udn.vn}

\subjclass[2010]{Primary 14M15, 14N35; Secondary 05E05, 55N91}

\keywords{Equivariant cohomology; Gromov-Witten invariant; Lagrangian Grassmannian; interpolation; Schubert structure constant; symmetric polynomial; quantum cohomology}

\date{\today}

\begin{abstract}
We first prove an identity involving symmetric polynomials. This identity leads us into exploring the geometry of Lagrangian Grassmannians. As an insight applications, we obtain a formula for the integral over the Lagrangian Grassmannian of a characteristic class of the tautological sub-bundle. Moreover, a relation to that over the ordinary Grassmannian and its application to the degree formula for the Lagrangian Grassmannian are given. Finally, we present further applications to the computation of Schubert structure constants and three-point, degree $1$, genus $0$ Gromov--Witten invariants of the Lagrangian Grassmannian. Some examples together with explicit computations are presented.
\end{abstract}

\maketitle

\section{Introduction}

Throughout we always assume that all polynomials are over the field of rational numbers or any field of characteristic zero. For convenience, we shall write $[n]$ for the set $\{1,2,\ldots,n\}$. Let $\lambda_1,\ldots,\lambda_n$ be $n$ values such that $\lambda_i^2 \neq \lambda_j^2$ for all $i \neq j$ and $\lambda_i \neq 0$ for all $i$. For each $I \subseteq [n]$ and $i\in[n]$, we denote by 
$$\lambda_{i,I} = \left\{\begin{matrix}
\lambda_i & \text{ if } & i \in I\\
-\lambda_i & \text{ if } & i \not \in I
\end{matrix}\right. \text{ and } \lambda_I = (\lambda_{1,I}, \ldots, \lambda_{n,I}).$$
Recall that a polynomial $P(x_1,\ldots,x_n)$ in $n$ variables $x_1,\ldots,x_n$ is said to be {\it symmetric} if it is invariant under permutations of $x_1,\ldots,x_n$. The starting point of this paper is an identity involving symmetric polynomials. This identity leads us into exploring the geometry of Lagrangian Grassmannians.

\begin{theorem}\label{theo1}
Let $P(x_1,\ldots,x_n)$ be a symmetric polynomial of degree not greater than $\frac{n(n+1)}{2}$ in $n$ variables $x_1,\ldots,x_n$. Then the sum
\begin{equation}\label{la}
\sum_{I\subseteq [n]}\frac{P(\lambda_I)}{\displaystyle\prod_{i<j }(\lambda_{i,I} + \lambda_{j,I})\prod_{i=1}^n\lambda_{i,I}} = \frac{2^nc(n)}{n!},
\end{equation}
where $c(n)$ is the coefficient of $x_1^{2n-1}\ldots x_n^{2n-1}$ in the polynomial 
$$P(x_1,\ldots,x_n)\prod_{i\neq j}(x_i-x_j)\prod_{i<j}(x_i+x_j).$$ 
\end{theorem}

\begin{remark}
Note that the identity in Theorem \ref{theo1} can be rewritten as follows:
$$\sum_{I\subseteq [n]}\frac{(-1)^{|I|}P(\lambda_I)}{\displaystyle\prod_{i<j }(\lambda_{i,I} + \lambda_{j,I})} = \frac{2^nc(n)}{n!}\prod_{i=1}^n\lambda_i,$$
where $|I|$ means the number of elements in $I$. If the degree of $P(x_1,\ldots,x_n)$ is less than $\frac{n(n+1)}{2}$, then the sum (\ref{la}) is equal to zero.
\end{remark}

The identity in Theorem \ref{theo1} provides an effective way of dealing with integrals over Lagrangian Grassmannians. The idea is as follows. The localization formula in equivariant cohomology allows us to express integrals in terms of some data attached to the fixed points of a torus action. In particular, for the Grassmannians, we obtain interesting expressions involving rational polynomials. In the previous work \cite{H}, the author has discovered the case of classical Grassmannians. Continuously, this paper is related to the integral formula for Lagrangian Grassmannians. In order to state the result, we first recall the definition of a Lagrangian Grassmannian.

\begin{definition}
Let $V$ be a complex vector space of dimension $2n$ endowed with a symplectic (i.e., non-degenerate, skew-symmetric, bilinear) form. We denote by $LG(n)$ the {\it Lagrangian Grassmannian} parametrizing Lagrangian (i.e, maximal isotropic) subspaces of $V$.
\end{definition}
Note that $LG(n)$ is a smooth subvariety of the ordinary Grassmannian $G(n,2n)$ of $V$. Its dimension is
$$\dim(LG(n)) = \frac{n(n+1)}{2}.$$
The cohomology ring of $LG(n)$, which agrees with the ring of Schur $\tilde{Q}$--polynomials, was discovered by Pragacz \cite{P}, Pragacz--Ratajski \cite{PR}. For combinatorial aspects of this theory, we refer to the work of Boe--Hiller \cite{BH} and Stembridge \cite{S}. Recently, the study of Lagrangian Grassmannians has been focused on exploring the theory in equivariant and quantum setting. For the equivariant setting, we refer to the work of Ikeda \cite{I} and Ikeda--Naruse \cite{IN}. For the quantum setting, we refer to the work of Buch-Kresch-Tamvakis \cite{BKT1} and Kresch--Tamvakis \cite{KT1}. For more further results in this direction, we refer to the recent works of Buch--Kresch--Tamvakis \cite{BKT2, BKT3} and Ikeda-Mihalcea-Naruse \cite{IMN}.

We are mainly interested in the integration over the Lagrangian Grassmannian. More precisely, let us consider the following integral:
$$\int_{LG(n)}\Phi(\mathcal S),$$
where $\Phi(\mathcal S)$ is a characteristic class of the tautological sub-bundle $\mathcal S$ on $LG(n)$. Using the localization formula in equivariant cohomology, Zielenkiewicz \cite{Z} presented a way of expressing the integral as an iterated residue at infinity of a holomorphic function. As an insightful application, the identity in Theorem \ref{theo1} provides an effective way of handling such an expression.

\begin{theorem}\label{integral1}
Suppose that $\Phi(\mathcal S)$ is represented by a symmetric polynomial $P(x_1,\ldots,x_n)$ of degree not greater than $\frac{n(n+1)}{2}$ in $n$ variables $x_1,\ldots,x_n$ which are the Chern roots of the tautological sub-bundle $\mathcal S$ on $LG(n)$. Then the integral 
\begin{equation}\label{intfor}
\int_{LG(n)}\Phi(\mathcal S) = (-1)^{\frac{n(n+1)}{2}}\frac{c(n)}{n!},
\end{equation}
where $c(n)$ is the coefficient of $x_1^{2n-1}\cdots x_n^{2n-1}$ in the polynomial 
$$P(x_1,\ldots,x_n)\prod_{i\neq j}(x_i-x_j)\prod_{i<j}(x_i+x_j).$$
\end{theorem}

After the first version of the paper was completed, Tamvakis informed us that in general Darondeau-Pragacz \cite{DP} obtain an integral formula for the symplectic Grassmann bundle $F^{\omega}(n)(E) \to X$ of a rank $2n$ symplectic vector bundle $E$. In particular, if $X$ is a point and $E$ is trivial, then the integral formula of Darondeau-Pragacz says that the integral $\int_{LG(n)}\Phi(\mathcal S)$ is equal to the coefficient of the monomial $x_1^{2n-1}x_2^{2n-2}\cdots x_n^{n}$ in the polynomial
$$P(x_1,\ldots,x_n)\prod_{i< j}(x_i-x_j)\prod_{i<j}(x_i+x_j).$$
It seems that the statement of Theorem \ref{integral1} is similar to the special case of Darondeau-Pragacz. However, here we consider the coefficient of the monomial $x_1^{2n-1}\cdots x_n^{2n-1}$ in the product of the polynomial 
$$P(x_1,\ldots,x_n)\prod_{i<j}(x_i+x_j)$$
by the discriminant 
$$\Delta = \prod_{i\neq j}(x_i-x_j)$$
instead of that of the monomial $x_1^{2n-1}x_2^{2n-2}\cdots x_n^{n}$ in the product of the same polynomial by the Vandermonde determinant  
$$a_{\delta} = \prod_{i<j}(x_i-x_j).$$
Note that the sign $(-1)^{\frac{n(n+1)}{2}}$ appears in the right-hand side of the formula (\ref{intfor}) is due to the variables $x_1,\ldots,x_n$ are assumed to be the Chern roots of the tautological sub-bundle $\mathcal S$ rather than the dual of $\mathcal S$.

Furthermore, our approach is completely different from that of Darondeau-Pragacz and probably close to that of B\'erczi-Szenes \cite{BS}, Tu \cite{T} and Zielenkiewicz \cite{Z}. We apply the localization formula and the identity in Theorem \ref{theo1} using instead of the residue at infinity. The identity is inspired by the previous work of the author \cite{H} and the interpolation formula for symmetric polynomials due to Chen-Louck \cite{CL}. There are other approaches to integral formulas in general:
\begin{itemize}
\item Using Grothendieck residues, we refer to the work of Akyildiz-Carrell \cite{AC}.
\item Using symmetrizing operators, we refer to the work of Brion \cite{B}.
\end{itemize}
Combining the integral formula in Theorem \ref{integral1} with that in \cite{H}, we are able to relate two integral formulas directly. This leads us into exploring further applications to the geometry of Lagrangian Grassmannians. More concretely, we obtain a degree formula for $LG(n)$ which is simply the degree formula of a Schubert variety with respect to the partition $(n-1,\ldots,1,0)$ in the ordinary Grassmannian $G(n,2n)$ (see Proposition \ref{degree}). Combining the Gaimbelli formula for $LG(n)$ due to Pragacz \cite{P} with the relation, we show that the Schubert structure constants, which occur in the cup product expansion in the cohomology ring of $LG(n)$, can be expressed as intersection numbers on the ordinary Grassmannian $G(n,2n)$ (see Theorem \ref{structure}). Similarly, by \cite[Proposition 4]{KT1}, we show that the three-point, degree 1, genus 0 Gromov-Witten invariants on $LG(n)$ can be expressed as intersection numbers on the classical Grassmannian $G(n+1,2n+2)$ (see Theorem \ref{gw}). 

The rest of the paper is organized as follows: The identity is proved in Section 2. Section 3 is to give a review of equivariant cohomology and prove the integral formula. In Section 4, we present the relation between the integral over the Lagrangian Grassmannian and that over the ordinary Grassmannian. The degree formula for $LG(n)$ is given in Section 5. In the last section, we present further applications related to the Schubert structure constants and three-point, degree 1, genus 0 Gromov-Witten invariants. Some examples together with the computation illustrated by \textsc{Singular} are also presented in the last section.
 
\section{Proof of the identity}

The key ingredient of the proof is two following lemmas.

\begin{lemma}\label{lem1}
Let $Q(x)$ be a monic polynomial of degree $d+1$ with distinct roots. Then the expression
$$\sum_{Q(\alpha)=0}\frac{\alpha^r}{Q'(\alpha)} = \begin{cases}
0 \quad \text{ if } \quad 0 \leq r < d\\ 
1 \quad \text{ if } \quad r=d
\end{cases}.$$
\end{lemma}
\begin{proof}
Let $\gamma_0, \gamma_1, \ldots, \gamma_d$ be $d+1$ distinct roots of $Q(x)$. We then have 
$$\sum_{Q(\alpha)=0}\frac{\alpha^r}{Q'(\alpha)} = \sum_{i=0}^d\frac{\gamma_i^r}{\displaystyle\prod_{i\neq j}(\gamma_i-\gamma_j)}.$$
By the Lagrange interpolation formula for $x^r$ at $\gamma_0, \gamma_1, \ldots, \gamma_d$, we get
$$x^r = \sum_{i=0}^d\frac{\gamma_i^r\displaystyle\prod_{i\neq j}(x-\gamma_j)}{\displaystyle\prod_{i\neq j}(\gamma_i-\gamma_j)}.$$
This implies that
$$\sum_{i=0}^d\frac{\gamma_i^r}{\displaystyle\prod_{i\neq j}(\gamma_i-\gamma_j)} = \text{ the coefficient of } x^d = \begin{cases}
0 \quad \text{ if } \quad 0 \leq r < d\\ 
1 \quad \text{ if } \quad r=d
\end{cases}.$$
\end{proof}

\begin{lemma}\label{lem2}
Let $F(x_1,\ldots,x_n)$ be a polynomial in $n$ variables of degree not greater than $d_1+\cdots+d_n$ and $Q_1(x),\ldots,Q_n(x)$ be monic polynomials of degrees $d_1+1,\ldots,d_n+1$ with distinct roots. Then the expression
\begin{equation}\label{id1}
\sum_{Q_1(\alpha_1)=\cdots=Q_n(\alpha_n)=0}\frac{F(\alpha_1,\ldots,\alpha_n)}{Q'_1(\alpha_1)\cdots Q'_n(\alpha_n)}
\end{equation}
is independent of all the $Q_i$ and is equal to the coefficient of $x_1^{d_1}\cdots x_n^{d_n}$ in $F(x_1,\ldots,x_n)$.
\end{lemma}

\begin{proof}
By linearity, it is enough to consider monomials $F(x_1,\ldots,x_n) = x_1^{r_1}\cdots x_n^{r_n}, r_1+\cdots+r_n \leq d_1+\cdots+d_n$. The expression (\ref{id1}) factors as
$$\left(\sum_{Q_1(\alpha_1)=0}\frac{\alpha_1^{r_1}}{Q'_1(\alpha_1)}\right)\cdots \left(\sum_{Q_n(\alpha_n)=0}\frac{\alpha_n^{r_1}}{Q'_n(\alpha_n)}\right).$$
By Lemma \ref{lem1}, if $r_i = d_i$ for all $i$, then the product will be equal to $1$, else the product will be equal to $0$. The lemma follows.
\end{proof}

\begin{proof}[Proof of Theorem \ref{theo1}]
In order to prove Theorem \ref{theo1}, we apply Lemma \ref{lem2} to 
$$F(x_1,\ldots,x_n) = P(x_1,\ldots,x_n)\prod_{i\neq j}(x_i-x_j)\prod_{i<j}(x_i+x_j).$$
and 
$$Q_1(x) = \cdots = Q_n(x) = Q(x) = \prod_{i=1}^n(x-\lambda_i)(x+\lambda_i).$$
Note that $F(x_1,\ldots,x_n)$ is a symmetric polynomial of degree not greater than $n(2n-1)$ and $Q(x)$ is a monic polynomial of degree $2n$ with distinct roots $\lambda_1, \ldots,\lambda_n,-\lambda_1,\ldots,-\lambda_n$. In this case, the expression (\ref{id1}) can be reduced to be
\begin{equation}\label{id2}
n!\sum_{I\subseteq [n]}\frac{F(\lambda_I)}{Q'(\lambda_{1,I})\cdots Q'(\lambda_{n,I})}.
\end{equation}
For each $I\subseteq [n]$, we have
$$F(\lambda_I) = P(\lambda_I)\prod_{i\neq j}(\lambda_{i,I}-\lambda_{j,I})\prod_{i<j}(\lambda_{i,I}+\lambda_{j,I})$$
and
\begin{eqnarray*}
Q'(\lambda_{1,I})\cdots Q'(\lambda_{n,I}) & = & \prod_{i=1}^n\left(2\lambda_{i,I}\prod_{i\neq j}(\lambda_i-\lambda_j)(\lambda_i+\lambda_j)\right)\\
& = & 2^n \left(\prod_{i=1}^n\lambda_{i,I}\right)\left(\prod_{i\neq j}(\lambda_{i,I}-\lambda_{j,I})\right)\left(\prod_{i\neq j}(\lambda_{i,I}+\lambda_{j,I})\right)\\
&=& 2^n \left(\prod_{i=1}^n\lambda_{i,I}\right)\left(\prod_{i\neq j}(\lambda_{i,I}-\lambda_{j,I})\right)\left(\prod_{i<j}(\lambda_{i,I}+\lambda_{j,I})\right)^2.
\end{eqnarray*}
Thus the expression (\ref{id2}) can be reduced to be
$$\frac{n!}{2^n}\sum_{I\subseteq [n]}\frac{P(\lambda_I)}{\displaystyle\prod_{i<j}(\lambda_{i,I} + \lambda_{j,I})\prod_{i=1}^n\lambda_{i,I}}.$$
The identity follows.
\end{proof}

\begin{remark}
The argument of the proof in this section is more compact and applicable than that of \cite[Section 2]{H}. The original idea seems to be similar. For another proof of \cite[Theorem 1]{H}, we apply Lemma \ref{lem2} to  
$$F(x_1,\ldots,x_k) = P(x_1,\ldots,x_k)\prod_{i\neq j}(x_i-x_j)$$
and
$$Q_1(x)=\cdots = Q_k(x) = Q(x) = \prod_{i=1}^n(x-\lambda_i).$$
\end{remark}

\section{Localization in equivariant cohomology}

In this section, we recall some basic definitions and important results in the theory of equivariant cohomology. For more details on this theory, we refer to \cite{Br}. Throughout we consider all cohomologies with coefficients in the complex field $\mathbb C$.

Let $T = (\mathbb C^*)^n$ be an algebraic torus of dimension $n$, classified by the principal $T$-bundle $ET \to BT$, whose total space $ET$ is contractible. Let $X$ be a compact space endowed with a $T$-action. Put $X_T = X\times_T ET$, which is itself a bundle over $BT$ with fiber $X$. Recall that the $T$-equivariant cohomology of $X$ is defined to be $H^*_{T}(X) = H^*(X_T)$, where $H^*(X_T)$ is the ordinary cohomology of $X_T$. 
 
A $T$-equivariant vector bundle is a vector bundle $\mathcal E$ on $X$ together with a lifting of the action on $X$ to an action on $\mathcal E$ which is linear on fibers. Note that $\mathcal E_T$ is a vector bundle over $X_T$. The $T$-equivariant characteristic class $c^T(\mathcal E) \in H^*_T(X)$ is defined to be the characteristic class $c(\mathcal E_T)$.

Let $\chi(T)$ be the character group of the torus $T$. For each $\rho \in \chi(T)$, let $\mathbb C_{\rho}$ denote the one-dimensional representation of $T$ determined by $\rho$. Then $L_{\rho} = (\mathbb C_{\rho})_T$ is a line bundle over $BT$, and the assignment $\rho \mapsto -c_1(L_{\rho})$ defines an group isomorphism $f: \chi(T) \simeq H^2(BT)$, which induces a ring isomorphism $\Sym(\chi(T)) \simeq H^*(BT)$. We call $f(\rho)$ the weight of $\rho$. In particular, we denote by $\lambda_i$ the weight of $\rho_i$ defined by $\rho_i(x_1,\ldots,x_n) = x_i$. We thus obtain an isomorphism
$$H^*_T(\pt) = H^*(BT) \cong \mathbb C[\lambda_1,\ldots,\lambda_n].$$

Let $X^T$ be the fixed point locus of the torus action. An important result in equivariant cohomology, which is called the localization theorem, says that 
the inclusion $i : X^T \hookrightarrow X$ induces an isomorphism
$$i^* : H^*_T(X) \otimes_{\mathbb C[\lambda_1,\ldots,\lambda_n]} \mathcal R_T \simeq H^*_T(X^T)\otimes_{\mathbb C[\lambda_1,\ldots,\lambda_n]} \mathcal R_T,$$
where $\mathcal R_T \cong \mathbb C(\lambda_1,\ldots,\lambda_n)$ is the fraction field of $\mathbb C[\lambda_1,\ldots,\lambda_n]$. Historically, the  localization theorem was first studied by Borel \cite{Bo} and then further investigated by Quillen \cite{Q}. Independently, Atiyah-Bott \cite{AB} and Berline-Vergne \cite{BV} gave an explicit formula for the inverse isomorphism, provided that $X$ is a compact manifold and $X^T$ is finite. More precisely, we have the following formula.

\begin{theorem} [Atiyah-Bott \cite{AB}, Berline-Vergne \cite{BV}]
Suppose that $X$ is a compact manifold endowed with a torus action and the fixed point locus $X^T$ is finite. For $\alpha \in H^*_T(X)$, we have
\begin{equation}\label{ABBV}
\int_X\alpha = \sum_{p\in X^T}\frac{\alpha|_p}{e_p},
\end{equation}
where $e_p$ is the $T$-equivariant Euler class of the tangent bundle at the fixed point $p$, and $\alpha|_p$ is the restriction of $\alpha$ to the point $p$.
\end{theorem}

For many applications, the Atiyah--Bott--Berline--Vergne formula can be formulated in more down-to-earth terms. We are mainly interested in the computation of the integrals over $LG(n)$.

\begin{proof}[Proof of Theorem \ref{integral1}]
Consider the natural action of $T = (\mathbb C^*)^n$ on the complex vector space $V = \mathbb C^{2n}$ endowed with a standard symplectic form.
This induces a torus action on the Lagrangian Grassmannian $LG(n)$ with isolated fixed points. The fixed points of the torus action can be indexed by the subsets $I\subseteq [n]$:
$$p_I = Span\{e_i,e'_j : i \in I, j \not \in I\},$$
where $e_1,\ldots,e_n,e'_n,\ldots,e'_1$ are the coordinates on $\mathbb C^{2n}$. For each $p_I$, the weights of the torus action on the fiber $\mathcal S|_{p_I}$ are $\lambda_{i,I}, \ldots, \lambda_{n,I}$. By the symplectic form, the tangent bundle on $LG(n)$ can be identified with the second symmetric power $\Sym^2\mathcal S^\vee$ of the dual of $\mathcal S$. Thus the weights of the torus action on the tangent bundle at $p_I$ are
$$\{-\lambda_{i,I} - \lambda_{j,I} : 1 \leq i \leq j \leq n\}.$$
This implies that the $T$-equivariant Euler class of the tangent bundle at $p_I$ is 
$$e_{p_I} = \prod_{i\leq j}(-\lambda_{i,I}-\lambda_{j,I}) = (-1)^{\frac{n(n+1)}{2}}2^n\prod_{i<j}(\lambda_{i,I}+\lambda_{j,I})\prod_{i=1}^n\lambda_{i,I}.$$
By the assumption, the $T$-equivariant characteristic class at $p_I$ is  
$$\Phi^T(\mathcal S|_{p_I}) = P(\lambda_I).$$
By the Atiyah--Bott--Berline--Vergne formula (\ref{ABBV}), we have
\begin{align*}
\int_{LG(n)}\Phi(\mathcal S) & = \sum_{p_I}\frac{\Phi^T(\mathcal S|_{p_I})}{e_{p_I}}\\
& = \frac{(-1)^{\frac{n(n+1)}{2}}}{2^n}\sum_{I\subseteq[n]}\frac{P(\lambda_I)}{\displaystyle\prod_{i<j}(\lambda_{i,I}+\lambda_{j,I})\prod_{i=1}^n\lambda_{i,I}}.
\end{align*}
Combining this with Theorem \ref{theo1}, Theorem \ref{integral1} is proved as desired.
\end{proof}

\section{Relation to the ordinary Grassmannian}

Let $G(n,2n)$ be the Grassmannian parametrizing $n$-dimensional subspaces of a $2n$-dimensional complex vector space. We denote by $\sigma_{\delta_n}$ the Schubert class on $G(n,2n)$ with respect to the partition
$$\delta_n = (n-1,n-2,\ldots, 1).$$
It is easy to see that $\sigma_{\delta_n}$ is represented by the symmetric polynomial
$$(-1)^{\frac{n(n-1)}{2}}\prod_{i<j}(x_i+x_j).$$
Note that in this case $x_1,\ldots,x_n$ are the Chern roots of the tautological sub-bundle $\mathcal S$ on $G(n,2n)$. By abuse of notation, we denote by $\Psi(\mathcal S)$ a characteristic class of the tautological sub-bundle $\mathcal S$ on $G(n,2n)$. Let us consider the following integral
$$\int_{G(n,2n)}\Psi(\mathcal S) \sigma_{\delta_n}.$$
By \cite[Theorem 3]{H}, if $\Psi(\mathcal S)$ is represented by a symmetric polynomial $P(x_1,\ldots,x_n)$ of degree not greater than $\frac{n(n+1)}{2}$, then the integral
$$\int_{G(n,2n)}\Psi(\mathcal S)\sigma_{\delta_n} = (-1)^{\frac{n(n+1)}{2}}\frac{c(n)}{n!},$$
where $c(n)$ is the coefficient of $x_1^{2n-1}\cdots x_n^{2n-1}$ in the polynomial 
$$P(x_1,\ldots,x_n)\prod_{i\neq j}(x_i-x_j)\prod_{i<j}(x_i+x_j).$$
Combining this with Theorem \ref{integral1}, we have the following relation.

\begin{proposition}\label{relation1}
Suppose that both $\Phi(\mathcal S)$ and $\Psi(\mathcal S)$ are represented by the same symmetric polynomial $P(x_1,\ldots,x_n)$ of degree not greater than $\frac{n(n+1)}{2}$. Then we have
$$\int_{LG(n)}\Phi(\mathcal S) = \int_{G(n,2n)}\Psi(\mathcal S)\sigma_{\delta_n}.$$
\end{proposition}

The relation allows us to reduce the computation of the integral over $LG(n)$ into the computation of that over $G(n,2n)$. By abuse of notation, we denote by the $\sigma_i = c_i(\mathcal S^\vee)$ the $i$-th Chern classes of the dual of the tautological sub-bundles $\mathcal S$ on both $LG(n)$ and $G(n,2n)$. It is well-known that the integral cohomology rings of $LG(n)$ and $G(n,2n)$ are presented  as the quotient of $\mathbb Z[\sigma_1,\ldots,\sigma_n]$ by the relations coming from the corresponding Whitney sum formulas (see for example \cite{P}). Note that these rings are graded with the degree of $\sigma_i$ is equal to $i$. We have the following corollary.

\begin{corollary}\label{cor1}
Let $P(\sigma_1,\ldots,\sigma_n)$ be a characteristic class on $LG(n)$. Then we have
$$\int_{LG(n)}P(\sigma_1,\ldots,\sigma_n) = \int_{G(n,2n)}P(\sigma_1,\ldots,\sigma_n)\sigma_{\delta_n},$$
where the latter class $P(\sigma_1,\ldots,\sigma_n)$ is on $G(n,2n)$.
\end{corollary}

\begin{remark}
For more details on the computation of the integral over the ordinary Grassmannian $G(n,2n)$, we refer to \cite[Chapter 2 and Chapter 5]{H2}. An implementation of the computation was made in \textsc{Singular} \cite{H3}.
\end{remark}

\begin{example}
The integral
$$\int_{LG(3)}\sigma^2_1\sigma^2_2 = \int_{G(3,6)}\sigma^2_1\sigma^2_2\sigma_{2,1} = 4.$$
The computation is illustrated by \textsc{Singular} as follows:
\begin{verbatim}
    > LIB "schubert.lib";
    > variety G36 = Grassmannian(3,6);
    > def r = G36.baseRing; setring r;
    > poly S1 = SchubertClass(list(1));
    > poly S2 = SchubertClass(list(2));
    > poly S21 = SchubertClass(list(2,1));
    > integral(G36,S1^2*S2^2*S21);
    4
\end{verbatim}
\end{example}

\section{Degree formula for the Lagrangian Grassmannian}

In this section, we apply Theorem \ref{relation1} to explore a degree formula for $LG(n)$. Recall that the degree of a subvariety of a projective space is defined to be the number of its intersection points with a generic linear subspace of complementary dimension. The degree of $LG(n)$, considered as a subvariety of a projective space thanks to the Pl\"ucker embedding, is given by the following formula:
$$\deg(LG(n)) = \int_{LG(n)}\sigma_1^{\frac{n(n+1)}{2}}.$$
By Corollary \ref{cor1}, we have
$$\int_{LG(n)}\sigma_1^{\frac{n(n+1)}{2}} = \int_{G(n,2n)}\sigma_1^{\frac{n(n+1)}{2}}\sigma_{\delta_n}.$$
Thus, by \cite[Example 14.7.11]{F}, we obtain a degree formula for $LG(n)$.

\begin{proposition} \label{degree}
The degree of a Lagrangian Grassmannian is given by
$$\deg(LG(n))= \frac{\frac{n(n+1)}{2}!}{\displaystyle \prod_{i=1}^n(2i-1)!}\prod_{1 \leq i<j\leq n}(2j-2i).$$
\end{proposition}

\begin{example}
By the degree formula, the degree of $LG(3)$ is $16$. It can be computed directly as follows:
$$\deg(LG(3)) = \int_{G(3,6)}\sigma_1^6\sigma_{2,1} = 16.$$
The computation is illustrated by \textsc{Singular} as follows:
\begin{verbatim}
    > integral(G36,S1^6*S21);
    16
\end{verbatim}
\end{example}

\section{Geometry of the Lagrangian Grassmannian}

In this section, we present further applications related to the geometry of $LG(n)$. We first recall the definition of Schubert classes on $LG(n)$. For each positive integer $n$, we denote by $\mathcal D_n$ the set of strict partitions $\alpha = (\alpha_1 > \alpha_2, \cdots > \alpha_l > 0)$ with $\alpha_1 \leq n$. The number $l = l(\alpha)$ is the length of $\alpha$. Fix an isotropic flag of subspaces $F_i$ of $V$: $$0 \subset F_1 \subset F_2 \subset \cdots \subset F_n \subset V,$$ 
where $\dim(F_i) = i$ for all $i$ and $F_n$ is Lagrangian. For each strict partition $\alpha \in \mathcal D_n$, we define
$$X_\alpha = \{\Sigma \in LG(n) \mid \dim(\Sigma\cap F_{n+1-\alpha_i}) \geq i \text{ for } 1 \leq i \leq l(\alpha)\}$$
to be a {\it Schubert variety}. This is a subvariety of codimension $|\alpha| = \sum\alpha_i$ in $LG(n)$. The Poincar\'e class $\sigma_\alpha = [X_\alpha] \in H^{2|\alpha|}(LG(n),\mathbb Z)$, does not depend on the choice of the flag, is called a {\it Schubert class}. The classes $\sigma_i = c_i(\mathcal S^\vee)$, for $i = 1, \ldots, n$, are called {\it special Schubert classes}. According to Pragacz \cite{P}, the Giambelli formula for $LG(n)$ is obtained in two stages as follows:

\begin{itemize}
\item For $i > j >0$, we have
$$\sigma_{i,j} = \sigma_i\sigma_j + 2 \sum_{k=1}^{\min\{n-i,j\}}(-1)^k\sigma_{i+k}\sigma_{j-k}.$$
For convenience, we set $\sigma_0 = 1$. 
\item For any $\alpha \in \mathcal D_n$ of length at least $3$, we have 
\begin{equation}\label{Pfaffian}
\sigma_\alpha = \Pfaffian[\sigma_{\alpha_i,\alpha_j}]_{1\leq i < j \leq r},
\end{equation}
where $r$ is the least even integer which is greater than or equal to $l(\alpha)$. If $l(\alpha)$ is odd then we set $\alpha_r =0$ and $\sigma_{\alpha_i,0} = \sigma_{\alpha_i}$.
\end{itemize}

\begin{remark}
The Pfaffian formula (\ref{Pfaffian}) is equivalent to the Laplace-type expansion for Pfaffians
$$\sigma_\alpha = \sum_{k=1}^{r-1}(-1)^{k-1}\sigma_{\alpha_k,\alpha_r}\sigma_{\alpha\setminus\{\alpha_k,\alpha_r\}}.$$
\end{remark}
Note that the set $\{\sigma_\alpha : \alpha\in\mathcal D_n\}$ forms a linear basis for the cohomology ring $H^*(LG(n),\mathbb Z)$ (see for example \cite{P}). We denote by $e_{\alpha,\beta}^\gamma$ the {\it Schubert structure constant} which occurs in the cup product expansion
\begin{equation}\label{sc}
\sigma_\alpha\sigma_\beta = \sum_{|\gamma|=|\alpha|+|\beta|}e_{\alpha,\beta}^\gamma\sigma_\gamma
\end{equation}
in $H^*(LG(n),\mathbb Z)$. For each strict partition $\alpha$, we denote by $\alpha^\vee$ the partition whose parts complement the parts of $\alpha$ in the set $\{1,\ldots,n\}$. Note that 
$$\int_{LG(n)}\sigma_\alpha\sigma_\beta = \left\{\begin{matrix}
1 & \text{ if } & \beta = \alpha^\vee\\
0 & \text{ if } & \beta \neq \alpha^\vee
\end{matrix}\right..$$

We denote by $QH^*(LG(n),\mathbb Z)$ the {\it small quantum cohomology ring} of $LG(n)$. This is a deformation of $H^*(LG(n),\mathbb Z)$ which first appeared in the work of string theorists. The ring $QH^*(LG(n),\mathbb Z)$ is an algebra over $\mathbb Z[q]$, where $q$ is a formal variable of degree $n+1$. By taking $q=0$, we recover the classical cohomology ring $H^*(LG(n),\mathbb Z)$.

\begin{definition}
In $QH^*(LG(n),\mathbb Z)$, the {\it quantum product} is given by the following formula:
$$\sigma_\alpha\cdot\sigma_\beta = \sum \langle\sigma_\alpha,\sigma_\beta,\sigma_{\gamma^\vee}\rangle_d \sigma_\gamma q^d,$$
where the sum runs over $d \geq 0$ and $\gamma\in\mathcal D_n$ with $|\gamma| +d(n+1) = |\alpha|+|\beta|$. The {\it quantum structure constant} $\langle\sigma_\alpha,\sigma_\beta,\sigma_{\gamma^\vee}\rangle_d$ is equal to the {\it three-point, genus $0$ Gromov--Witten invariant}, which counts the number of rational curves of degree $d$ meeting three Schubert varieties $X_\alpha,X_\beta$ and $X_{\gamma^\vee}$ in general position.
\end{definition}

Kresch-Tamvakis \cite[Theorem 1]{KT1} gave a presentation for $QH^*(LG(n),\mathbb Z)$ and showed that the Gromov--Witten invariant $\langle\sigma_\alpha,\sigma_\beta,\sigma_{\gamma^\vee}\rangle_d$ can be determined by the ring of $\tilde{Q}$-polynomials. Using the Giambelli formula for $LG(n)$ and Corollary \ref{cor1}, we give another way to compute the Schubert structure constant $e_{\alpha,\beta}^\gamma$ and the  Gromov--Witten invariant $\langle\sigma_\alpha,\sigma_\beta,\sigma_{\gamma^\vee}\rangle_1$. More precisely, we show that both the Schubert structure constant $e_{\alpha,\beta}^\gamma$ and the  Gromov--Witten invariant $\langle\sigma_\alpha,\sigma_\beta,\sigma_{\gamma^\vee}\rangle_1$ can be expressed as integrals over the classical Grassmannians.

\begin{definition}
For each $\alpha \in \mathcal D_n$, we define the class $\tilde{Q}_\alpha \in H^*(G(n,2n),\mathbb Z)$ as follows:
\begin{itemize}
\item $\tilde{Q}_i = \sigma_i$ for all $i = 1,\ldots,n$. 
\item For $i > j >0$, we define
$$\tilde{Q}_{i,j} = \tilde{Q}_i\tilde{Q}_j + 2 \sum_{k=1}^{\min\{n-i,j\}}(-1)^k\tilde{Q}_{i+k}\tilde{Q}_{j-k}.$$
For convenience, we set $\tilde{Q}_0 = 1$.
\item For any $\alpha \in \mathcal D_n$ of length at least $3$, we define
$$\tilde{Q}_\alpha = \Pfaffian[\tilde{Q}_{\alpha_i,\alpha_j}]_{1\leq i < j \leq r},$$
where $r$ is the least even integer which is greater than or equal to $l(\alpha)$. If $l(\alpha)$ is odd then we set $\alpha_r =0$ and $\tilde{Q}_{\alpha_i,0} = \tilde{Q}_{\alpha_i}$.
\end{itemize} 
\end{definition}

\begin{theorem}\label{structure}
Let $\alpha,\beta,\gamma \in \mathcal D_n$ such that $|\gamma| = |\alpha|+|\beta|$. The Schubert structure constant $e_{\alpha,\beta}^\gamma$ can be expressed as an intersection number on the ordinary Grassmannian $G(n,2n)$, that is
$$e_{\alpha,\beta}^\gamma = \langle\sigma_\alpha,\sigma_\beta,\sigma_{\gamma^\vee}\rangle_0 = \int_{G(n,2n)}\tilde{Q}_\alpha\tilde{Q}_\beta\tilde{Q}_{\gamma^\vee}\sigma_{\delta_n}.$$
\end{theorem}

\begin{proof}
By the definition and notation as above, the Schubert structure constant $e_{\alpha,\beta}^\gamma$ can be expressed as a triple intersection number
$$e_{\alpha,\beta}^\gamma = \langle\sigma_\alpha,\sigma_\beta,\sigma_{\gamma^\vee}\rangle_0 = \int_{LG(n)}\sigma_\alpha\sigma_\beta\sigma_{\gamma^\vee}.$$
Combining this with the Giambelli formula for $LG(n)$ and Corollary \ref{cor1}, the theorem follows.
\end{proof}

\begin{theorem}\label{gw}
Let $\alpha,\beta,\gamma \in \mathcal D_n$ such that $|\gamma| + n+1 = |\alpha|+|\beta|$. Then the Gromov--Witten invariant 
$$\langle\sigma_\alpha,\sigma_\beta,\sigma_{\gamma^\vee}\rangle_1 = \frac{1}{2}\int_{G(n+1,2n+2)}\tilde{Q}_\alpha\tilde{Q}_\beta\tilde{Q}_{\gamma^\vee}\sigma_{\delta_{n+1}},$$
where $\tilde{Q}_\alpha, \tilde{Q}_\beta$ and $\tilde{Q}_{\gamma^\vee}$ are determined as the classes on $G(n+1,2n+2)$.
\end{theorem}

\begin{proof}
By \cite[Proposition 4]{KT1}, the Gromov--Witten invariant 
$$\langle\sigma_\alpha,\sigma_\beta,\sigma_{\gamma^\vee}\rangle_1 = \frac{1}{2}\int_{LG(n+1)}\sigma_\alpha\sigma_\beta\sigma_{\gamma^\vee},$$
where $\sigma_\alpha, \sigma_\beta$ and $\sigma_{\gamma^\vee}$ denote the Schubert classes on $LG(n+1)$. Combining this with the Giambelli formula and Corollary \ref{cor1} for $LG(n+1)$, the theorem follows.
\end{proof}


\begin{example}
In $QH^*(LG(3),\mathbb Z)$, we have the quantum product
$$\sigma_{2,1}\cdot \sigma_2 = e_{(2,1),(2)}^{(3,2)} \sigma_{3,2} + \langle\sigma_{2,1},\sigma_2,\sigma_{3,2}\rangle_1 \sigma_1q.$$
The Schubert structure constant 
\begin{eqnarray*}
e_{(2,1),(2)}^{(3,2)} = \langle\sigma_{2,1},\sigma_2,\sigma_1\rangle_0 & = & \int_{G(3,6)}\tilde{Q}_{2,1}\tilde{Q}_2\tilde{Q}_1\sigma_{2,1},
\end{eqnarray*}
which is computed by \textsc{Singular} as follows:
\begin{verbatim}
    > LIB "schubert.lib";
    > variety G = Grassmannian(3,6);
    > def r = G.baseRing; setring r;
    > poly Q1 = QtildeClass(list(1));
    > poly Q2 = QtildeClass(list(2));
    > poly Q21 = QtildeClass(list(2,1));
    > poly S21 = SchubertClass(list(2,1));
    > integral(G,Q21*Q2*Q1*S21);
    2
\end{verbatim}
It follows that $e_{(2,1),(2)}^{(3,2)} = 2$. 
The Gromov--Witten invariant
\begin{eqnarray*}
\langle\sigma_{2,1},\sigma_2,\sigma_{3,2}\rangle_1 & = & \frac{1}{2}\int_{G(4,8)}\tilde{Q}_{2,1}\tilde{Q}_2\tilde{Q}_{3,2}\sigma_{3,2,1},
\end{eqnarray*}
which is computed by \textsc{Singular} as follows:
\begin{verbatim}
    > variety G = Grassmannian(4,8);
    > def r = G.baseRing; setring r;
    > poly Q21 = QtildeClass(list(2,1));
    > poly Q2 = QtildeClass(list(2));
    > poly Q32 = QtildeClass(list(3,2));
    > poly S321 = SchubertClass(list(3,2,1));
    > integral(G,Q21*Q2*Q32*S321);
    2
\end{verbatim}
This follows that $\langle\sigma_{2,1},\sigma_2,\sigma_{3,2}\rangle_1 = 1$. 
Thus the quantum product
$$\sigma_{2,1}\cdot \sigma_2 = 2 \sigma_{3,2} + \sigma_1q.$$
This agrees with the quantum Pieri rule given by Buch-Kresch-Tamvakis \cite[Theorem 3]{BKT1}.
\end{example}

\begin{example}
In $QH^*(LG(4),\mathbb Z)$, we have the quantum product
$$\sigma_{3,2}\cdot \sigma_{2,1} = e_{(3,2),(2,1)}^{(4,3,1)} \sigma_{4,3,1}  + \langle\sigma_{3,2},\sigma_{2,1},\sigma_{4,2,1}\rangle_1 \sigma_3q + \langle\sigma_{3,2},\sigma_{2,1},\sigma_{4,3}\rangle_1 \sigma_{2,1}q.$$
The Schubert structure constant $e_{(3,2),(2,1)}^{(4,3,1)}$ is computed as follows:
\begin{eqnarray*}
e_{(3,2),(2,1)}^{(4,3,1)} = \langle\sigma_{3,2},\sigma_{2,1},\sigma_2\rangle_0 & = & \int_{G(4,8)}\tilde{Q}_{3,2}\tilde{Q}_{2,1}\tilde{Q}_2\sigma_{3,2,1}.
\end{eqnarray*}
The integral is computed by \textsc{Singular} as follows:
\begin{verbatim}
    > variety G = Grassmannian(4,8);
    > def r = G.baseRing; setring r;
    > poly Q2 = QtildeClass(list(2));
    > poly Q21 = QtildeClass(list(2,1));
    > poly Q32 = QtildeClass(list(3,2));
    > poly S321 = SchubertClass(list(3,2,1));
    > integral(G,Q32*Q21*Q2*S321);
    2
\end{verbatim}
This means that $e_{(3,2),(2,1)}^{(4,3,1)} = 2$.\\
The Gromov--Witten invariants $\langle\sigma_{3,2},\sigma_{2,1},\sigma_{4,2,1}\rangle_1$ and $\langle\sigma_{3,2},\sigma_{2,1},\sigma_{4,3}\rangle_1$ are computed as follows:
\begin{eqnarray*}
\langle\sigma_{3,2},\sigma_{2,1},\sigma_{4,2,1}\rangle_1 & = & \frac{1}{2}\int_{G(5,10)}\tilde{Q}_{3,2}\tilde{Q}_{2,1}\tilde{Q}_{4,2,1}\sigma_{4,3,2,1}
\end{eqnarray*}
and 
\begin{eqnarray*}
\langle\sigma_{3,2},\sigma_{2,1},\sigma_{4,3}\rangle_1 & = & \frac{1}{2}\int_{G(5,10)}\tilde{Q}_{3,2}\tilde{Q}_{2,1}\tilde{Q}_{4,3}\sigma_{4,3,2,1}.
\end{eqnarray*}
These integrals are computed by \textsc{Singular} as follows:
\begin{verbatim}
    > variety G = Grassmannian(5,10);
    > def r = G.baseRing; setring r;
    > poly Q32 = QtildeClass(list(3,2));
    > poly Q21 = QtildeClass(list(2,1));
    > poly Q43 = QtildeClass(list(4,3));
    > poly Q421 = QtildeClass(list(4,2,1));
    > poly S4321 = SchubertClass(list(4,3,2,1));
    > integral(G,Q32*Q21*Q421*S4321);
    4
    > integral(G,Q32*Q21*Q43*S4321);
    2
\end{verbatim}
This means that $\langle\sigma_{3,2},\sigma_{2,1},\sigma_{4,2,1}\rangle_1 = 2$ and $\langle\sigma_{3,2},\sigma_{2,1},\sigma_{4,3}\rangle_1 = 1$. Thus, in $QH^*(LG(4),\mathbb Z)$, we have the quantum product
$$\sigma_{3,2}\cdot \sigma_{2,1} = 2 \sigma_{4,3,1} + 2 \sigma_3q + \sigma_{2,1}q.$$
\end{example}

\subsection*{Acknowledgements}
This paper was written while the first author visited Vietnam Institute for advanced study in Mathematics (VIASM), he would like to thank VIASM for the very kind support and hospitality. This research was partially supported by grants no. B2018.DNA.10 and B2019.DLA.03.

\end{document}